\documentclass[a4paper,leqno]{amsart}

\usepackage{amsmath}
\usepackage{amsfonts}
\usepackage{palatino}
\usepackage{a4wide}
\usepackage{amssymb}
\usepackage{amsthm}

\theoremstyle{plain}

\newtheorem{teo}{Theorem}[section]
\newtheorem{lemma}[teo]{Lemma}

\newtheorem{ackn}{Acknowledgments\!}

\theoremstyle{definition}

\theoremstyle{remark}

\newtheorem{rem}[teo]{Remark}

\numberwithin{equation}{section}

\def\SS{{{\mathbb S}}}

\def\RR{{\mathbb R}}

\def\RRR{{\mathrm R}}
\def\WWW{{\mathrm W}}

\def\HHH{{\mathrm H}}

\def\Ric{{\mathrm {Ric}}}

\def\CCC{{\mathrm C}}

\def\div{\operatornamewithlimits{div}\nolimits}

\title[Locally Conformally Flat Quasi--Einstein Manifolds]{Locally Conformally Flat Quasi--Einstein Manifolds}
\date{\today}

\author[Giovanni Catino]{Giovanni Catino}
\address[Giovanni Catino]{SISSA -- International School for Advanced Studies, Via Bonomea 265 , Trieste, Italy, 34136}
\email[G. Catino]{catino@sissa.it}

\author[Carlo Mantegazza]{Carlo Mantegazza}
\address[Carlo Mantegazza]{Scuola Normale Superiore di Pisa, P.za Cavalieri 7, Pisa, Italy, 56126}
\email[C. Mantegazza]{c.mantegazza@sns.it}

\author[Lorenzo Mazzieri]{Lorenzo Mazzieri}
\address[Lorenzo Mazzieri]{SISSA -- International School for Advanced Studies, Via Bonomea 265, Trieste, Italy, 34136}
\email[L. Mazzieri]{mazzieri@sissa.it}

\author[Michele Rimoldi]{Michele Rimoldi}
\address[Michele Rimoldi]{Dipartimento di Matematica, 
Universit\`a degli Studi di Milano, 
Via Saldini 50, Milano, Italy, 20133}
\email[M. Rimoldi]{michele.rimoldi@unimi.it}

\date{\today}

\begin{document}

\begin{abstract} In this paper we prove that any complete locally
  conformally flat quasi-Einstein manifold of dimension $n\geq 3$ is
  locally a warped product with $(n-1)$--dimensional fibers 
  of constant curvature. This result includes also the case of locally conformally flat gradient Ricci solitons.
\end{abstract}

\maketitle

\section{Introduction}

Let $(M^n,g)$, for $n\geq 3$, be a complete quasi--Einstein Riemannian manifold,
that is, there exist a smooth function $f:M^n\to\mathbb{R}$ and  two
constants $\mu,\lambda\in\RR$ such that
\begin{equation}\label{qem}
\Ric+\nabla^{2} f - \mu\, df \otimes df = \lambda g \,.
\end{equation}

When $\mu=0$, quasi--Einstein manifolds correspond to gradient Ricci
solitons and when $f$ is constant~\eqref{qem} gives the Einstein
equation and we call the quasi--Einstein metric trivial. We also
notice that, for $\mu=\tfrac{1}{2-n}$, the metric
$\widetilde{g}=e^{-\frac{2}{n-2}f}g$ is Einstein. Indeed, from the
expression of the Ricci tensor of a conformal metric, we get
\begin{eqnarray*}
\Ric_{\widetilde{g}}&=&\Ric_{g} +\nabla^{2} f +\tfrac{1}{n-2} df \otimes
df +\tfrac{1}{n-2} \big(\Delta f - |\nabla f|^{2} \big) g\\
&=& \, \tfrac{1}{n-2}\big(\Delta f - |\nabla f|^{2}+(n-2)\lambda\big)
e^{\frac{2}{n-2}f} \,\widetilde{g}\,.
\end{eqnarray*}
In particular, if $g$ is also locally conformally flat, then
$\widetilde{g}$ has constant curvature.

Quasi--Einstein manifolds have been recently introduced by J.~Case,
Y.-S.~Shu and G.~Wei in~\cite{caseshuwei}. In that work the authors focus mainly on the case $\mu\geq 0$. The case $\mu=\frac{1}{m}$ for some $m\in\mathbb{N}$ is particularly relevant due to the link with Einstein warped products. Indeed in~\cite{caseshuwei}, following the results in~\cite{kimkim}, it is proved a characterization of these quasi--Einstein
metrics as base metrics of Einstein warped product metrics 
(see also~\cite[Theorem 2]{rimoldi}). This characterization on the one hand enables
to translate results from one setting to the other and on the other
hand permits to furnish several examples of quasi--Einstein manifolds 
(see~\cite[Chapter 9]{Besse},~\cite{lupagepope}). Observe also that, in case $\frac{1}{2-n}\leq \mu<0$, the definition of quasi--Einstein metric was used by D. Chen in \cite{ChenD} in the context of finding conformally Einstein product metrics on $M^n\times F^m$ for $\mu=\frac{1}{2-m-n}$, $m\in\mathbb{N}\cup\left\{0\right\}$. 

As a generalization of Einstein manifolds, quasi--Einstein manifolds
exhibit a certain rigidity. This is well known for $\mu=0$, 
but we have evidence of this also in the case $\mu\geq 0$. This is expressed for example by 
triviality results and curvature estimates;
see~\cite{case,caseshuwei,rimoldi}. For instance it is known that,
according to Qian version of Myers' Theorem, 
if $\lambda>0$ and $\mu>0$ in~\eqref{qem} then $M^n$ is compact 
(see~\cite{qian}). Moreover in~\cite{kimkim}, analogously to the case $\mu=0$, it is proven that if $\lambda\leq0$, compact
quasi--Einstein manifolds are trivial. A generalization to the
complete non-compact setting of this result is obtained
in~\cite{rimoldi} by means of an $L^p$--Liouville result for the
weighted Laplacian which relies upon estimates for the infimum of the
scalar curvature (extending the previous work 
in~\cite{caseshuwei}). These latter are achieved by means of tools
coming from stochastic analysis such as the weak maximum principle at
infinity combined with Qian's estimates on weighted volumes 
(see \cite{qian} and also~\cite[Section 2]{maririgolisetti}).

\bigskip

The Riemann curvature
operator of a Riemannian manifold $(M^n,g)$ is defined 
as in~\cite{gahula} by
$$
\mathrm{Riem}(X,Y)Z=\nabla_{Y}\nabla_{X}Z-\nabla_{X}\nabla_{Y}Z+\nabla_{[X,Y]}Z\,.
$$ 
In a local coordinate system the components of the $(3,1)$--Riemann 
curvature tensor are given by
$\RRR^{d}_{abc}\tfrac{\partial}{\partial
  x^{d}}=\mathrm{Riem}\big(\tfrac{\partial}{\partial
  x^{a}},\tfrac{\partial}{\partial
  x^{b}}\big)\tfrac{\partial}{\partial x^{c}}$ and we denote by
$\RRR_{abcd}=g_{de}\RRR^{e}_{abc}$ its $(4,0)$--version.

\medskip

{\em In all the paper the Einstein convention of summing over the repeated 
indices will be adopted.}

\medskip

With this choice, for the sphere $\SS^n$ we have
${\mathrm{Riem}}(v,w,v,w)=\RRR_{abcd}v^aw^bv^cw^d>0$. The Ricci tensor is obtained by the contraction 
$\RRR_{ac}=g^{bd}\RRR_{abcd}$ and $\RRR=g^{ac}\RRR_{ac}$ will 
denote the scalar curvature. The so called Weyl tensor is then 
defined by the following decomposition formula (see~\cite[Chapter~3,
Section~K]{gahula}) in dimension $n\geq 3$,
\begin{equation}\label{Weyl}
\WWW_{abcd}=\,\RRR_{abcd}+\frac{\RRR}{(n-1)(n-2)}(g_{ac}g_{bd}-g_{ad}g_{bc})
- \frac{1}{n-2}(\RRR_{ac}g_{bd}-\RRR_{ad}g_{bc}
+\RRR_{bd}g_{ac}-\RRR_{bc}g_{ad}).
\end{equation}
The Weyl tensor satisfies all the symmetries of the curvature tensor
and all its traces with the metric are zero, 
as it can be easily seen by the above formula.\\
In dimension three $\WWW$ is identically zero for every Riemannian
manifold, it becomes relevant instead when $n\geq 4$ since its
vanishing is a condition equivalent for $(M^n,g)$ to be {\em locally
  conformally flat}, that is, around every point $p\in M^n$ there is a conformal deformation $\widetilde{g}_{ab}=e^fg_{ab}$ of the original metric $g$, such that the new metric is flat, 
namely, the Riemann tensor associated to $\widetilde{g}$ is zero in $U_p$ (here $f:U_p\to\RR$ is a smooth function defined in a open
neighborhood $U_p$ of $p$). In dimension $n=3$, on the other hand, locally conformally flatness is equivalent to the vanishing of the Cotton tensor
$$
\CCC_{abc} = \nabla_{c} \RRR_{ab} - \nabla_{b} \RRR_{ac} -
\tfrac{1}{2(n-1)} \big( \nabla_{c} \RRR \, g_{ab} - \nabla_{b} \RRR \,
g_{ac} \big)\,.
$$ 
When $n\geq 4$ note that one can compute, (see \cite{Besse}), that
\begin{equation*}
\,\nabla^d\WWW_{abcd}=-\frac{n-3}{n-2}\CCC_{abc}.
\end{equation*}
Hence if we assume that the manifold is locally conformally flat, the Cotton tensor is identically zero also in this case.

\medskip

In this paper we will consider a generic $\mu\in\RR$ and we will prove the following

\begin{teo}\label{main}
Let $(M^{n},g)$, $n\geq 3$, be a complete locally conformally flat quasi--Einstein manifold. Then
\begin{enumerate}
\item[(i)] if $\mu=\tfrac{1}{2-n}$, then $(M^{n},g)$ is globally conformally equivalent to a spaceform.  
\item[(ii)] if $\mu\neq\tfrac{1}{2-n}$, then, around any regular point
  of $f$, the manifold $(M^{n},g)$ is locally a warped product with
  $(n-1)$-dimensional fibers of constant sectional curvature.
\end{enumerate} 
\end{teo}

\begin{rem} 
This result was already known in the case where $\mu=0$, i.e. for
gradient Ricci solitons (see~\cite{caochen} and ~\cite{mancat1}).
Nevertheless, the strategy of the proof is completely new and can be
used as the main step to classify locally conformally flat shrinking and
steady gradient Ricci solitons.  
\end{rem}

\begin{rem}
Very recently, similar results have been obtained by C.~He, P.~Petersen and W.~Wylie~\cite{HePetWylie} in the case when $0<\mu<1$, assuming a slightly weaker condition than locally conformally flatness.
\end{rem}

\section{Proof of Theorem~\ref{main}}
As observed in the introduction,  if $\mu=\frac{1}{2-n}$, $(M^n,g)$ is globally conformally equivalent to a space form.

\medskip

{\em  From now on, we will consider the case $\mu\neq\tfrac{1}{2-n}$. }

\begin{lemma} Let $(M^{n},g)$ be a quasi--Einstein manifold. Then the following identities hold
\begin{eqnarray}
\label{eq1}&\RRR + \Delta f - \mu |\nabla f|^{2} = n\lambda& \\
\label{eq2}&\nabla_{b} \RRR = 2 \RRR_{ab} \nabla^{a} f +2 \mu R
\nabla_{b} f -2 \mu^{2} |\nabla f|^{2} \nabla_{b} f - 2 n\mu\lambda
\nabla_{b} f + \mu \nabla_{b} |\nabla f|^{2}& \\
\label{eq3}&\nabla_{c} \RRR_{ab} - \nabla_{b} \RRR_{ac} = -
\RRR_{cabd} \nabla^{d} f +\mu\big( \RRR_{ab} \nabla_{c} f - \RRR_{ac}
\nabla_{b} f\big) - \lambda\mu\big( g_{ab} \nabla_{c} f - g_{ac}
\nabla_{b} f \big)
\end{eqnarray} 
\end{lemma}

\begin{proof} Equation~\eqref{eq1}: we simply contract equation~\eqref{qem}. 

\medskip

\noindent Equation~\eqref{eq2}: we take the divergence of the equation~\eqref{qem}
\begin{eqnarray*}
\div\Ric_{b} & = & -\nabla^{a}\nabla_{a}\nabla_{b} f + \mu
\nabla^{a}\nabla_{a} f \nabla_{b} f + \mu \nabla^{a}\nabla_{b} f
\nabla_{a} f \\
& = & -\nabla_{b} \Delta f - \RRR_{ab} \nabla^{a} f + \mu \Delta f
\nabla_{b} f +\tfrac{1}{2}\mu \nabla_{b} |\nabla f|^{2}\,,
\end{eqnarray*}
where we interchanged the covariant derivatives. Now, using
equation~\eqref{eq1}, we get
\begin{eqnarray*}
\div\Ric_{b} & = & \nabla_{b} \RRR -\mu\nabla_{b} |\nabla
f|^{2} - \RRR_{ab} \nabla^{a} f -\mu \RRR \nabla_{b} f +\mu^{2}
|\nabla f|^{2} \nabla_{b} f + n\mu\lambda \nabla_{b} f +\tfrac{1}{2}
\mu \nabla_{b} |\nabla f|^{2} \\
& = & \nabla_{b} \RRR - \RRR_{ab} \nabla^{a} f -\mu R \nabla_{b} f
+\mu^{2} |\nabla f|^{2} \nabla_{b} f + n\mu\lambda \nabla_{b} f -
\tfrac{1}{2} \mu \nabla_{b} |\nabla f|^{2}
\end{eqnarray*}
Finally, using Schur's Lemma $\nabla \RRR = 2\div\Ric $, we obtain equation~\eqref{eq2}.

\medskip

\noindent Equation~\eqref{eq3}: taking the covariant derivative of~\eqref{qem} we obtain
\begin{eqnarray*} \nabla_{c} \RRR_{ab} -\nabla_{b} \RRR_{ac} &=& -
  \big(\nabla_{c}\nabla_{b}\nabla_{a} f -
  \nabla_{b}\nabla_{c}\nabla_{a} f\big) \\
&& +\,\mu\big(\nabla_{c}\nabla_{a} f \nabla_{b} f + \nabla_{c}
\nabla_{b} f \nabla_{a} f - \nabla_{b}\nabla_{a} f \nabla_{c} f -
\nabla_{b} \nabla_{c} f \nabla_{a} f\big)\\
& = & - \RRR_{cabd} \nabla^{d}f +\mu \big(\RRR_{ab} \nabla_{c} f -
\RRR_{ac} \nabla_{b} f\big) - \mu \lambda \big(g_{ab}\nabla_{c} f
-g_{ac}\nabla_{b} f\big)\,,   
\end{eqnarray*}
where we interchanged the covariant derivatives and we have used again
equation~\eqref{qem}.
\end{proof}

In any neighborhood, where $|\nabla f|\neq 0$, of a level set
$\Sigma_{\rho}=\{p\in M^n\,|\,\,\, f(p)=\rho\}$ of a regular value
$\rho$ of $f$, we can express the metric $g$ as
\begin{equation}\label{metric}
g = \tfrac{1}{|\nabla f|^{2}}\, df \otimes df + g_{ij}(f,\theta)
\,d\theta^{i} \otimes d\theta^{j}\,,
\end{equation}
where $\theta=(\theta^{1},\dots,\theta^{n-1})$ denotes intrinsic
coordinates for $\Sigma_{\rho}$. In the following computations we will
agree that $\partial_0=\frac{\partial}{\partial f}$, $\partial_j=\frac{\partial}{\partial \theta^j}$, $\nabla_{0} = \nabla_{\partial_0}$, 
$\RRR_{00}= \Ric(\partial_0,\partial_0)$, $\RRR_{0j}=
\Ric(\partial_0,\partial_j)$ and so on.
According to~\eqref{metric}, we compute easily that
\begin{equation}\label{compeasy}
\nabla_j f=0,\qquad \nabla_0 f=|\nabla f|^2, \qquad g^{00}=|\nabla f|^2\,.
\end{equation}
In this coordinate system, we have the following formulae:
\begin{lemma}\label{formulas} If $(M^{n},g)$ is a quasi--Einstein manifold. Then, for
  $j=1,\dots,n-1$, we have
\begin{eqnarray}
\label{id1}\nabla_{j} |\nabla f|^{2} &=& -2 |\nabla f|^{4}\RRR_{0j} \\
\label{id1.2}\nabla_{0} |\nabla f|^{2} &=& -2 |\nabla f|^{4}\RRR_{00}
+2\mu|\nabla f|^{4}+2\lambda|\nabla f|^{2}\\
\label{id2}\nabla_{j} \RRR &=& 2 (1-\mu\big) \, |\nabla
f|^{4}\RRR_{0j} \\
\label{id2.2}\nabla_{0} \RRR &=& 2 (1-\mu\big) \, |\nabla
f|^{4}\RRR_{00}  -2(n-1)\mu\lambda|\nabla f|^{2} + 2\mu \RRR |\nabla
f|^{2} \\
\label{id3}\nabla_{0} \RRR_{j0} - \nabla_{j} \RRR_{00} &=& \mu|\nabla
f|^{2}\,\RRR_{0j}\,.
\end{eqnarray}
Moreover, if $(M^{n},g)$ has $\WWW=0$, for $i,j=1,\dots,n-1$, we have
\begin{eqnarray}
\label{id3.2}\nabla_{0} \RRR_{ij} - \nabla_{j} \RRR_{i0} &=&
\tfrac{\mu(n-2)+1}{n-2}\RRR_{ij}|\nabla f|^{2} +
\tfrac{1}{n-2}\RRR_{00}|\nabla f|^{4} g_{ij} +\\
&&- \tfrac{1}{(n-1)(n-2)}\RRR |\nabla f|^{2} g_{ij} - \lambda \mu
|\nabla f|^{2} g_{ij}\nonumber\,.
\end{eqnarray}
\end{lemma}
\begin{proof} Equation~\eqref{id1}: we compute
\begin{equation*}
\nabla_{j} |\nabla f|^{2} = 2 \, g^{00}\nabla_{j}\nabla_{0} f
\nabla_{0} f + 2 g^{kl} \nabla_{j} \nabla_{k} f \nabla_{l} f. 
\end{equation*}
Using~\eqref{qem} and~\eqref{compeasy} we thus obtain
\begin{eqnarray*}
\nabla_{j} |\nabla f|^{2} &=& \, 2g^{00}\left(-\RRR_{0j}+\mu df\left(\partial_j\right)df\left(\partial_0\right)+\lambda g_{0j}\right)\nabla_0 f\\
&=&\,-2 |\nabla f|^{4} \RRR_{0j} +2\mu |\nabla f|^{2}\nabla_{j} f
\nabla_{0} f \nabla_{0} f + 2 \lambda |\nabla f|^{2}g_{0j} \nabla_{0} f \\
&=& \, -2 |\nabla f|^{4} \RRR_{0j}\,.
\end{eqnarray*}

\medskip

\noindent Equation~\eqref{id1.2}: we have as before,
\begin{eqnarray*}
\nabla_{0} |\nabla f|^{2} &=& 2 |\nabla f|^{4}\nabla^{2}_{00} f\\
&=& \, 2 |\nabla f|^{4}\left( -\RRR_{00} +\mu df(\partial_0)df(\partial_0) +  \lambda g_{00} \right) \\
&=& \, -2 |\nabla f|^{4} \RRR_{00} +2\mu |\nabla f|^{4} + 2 \lambda
|\nabla f|^{2}.
\end{eqnarray*}

\medskip

\noindent Equation~\eqref{id2}: using equation~\eqref{eq2} and
equation~\eqref{id1} one has
\begin{eqnarray*}
 \nabla_{j} \RRR = 2 |\nabla f|^{4} \RRR_{0j} -2\mu|\nabla f|^{4}
 \RRR_{0j}= 2(1-\mu) |\nabla f|^{2} \RRR_{0j}\,.
\end{eqnarray*}
\medskip

\noindent Equation~\eqref{id2.2}: it follows as before from
equations~\eqref{eq2} and~\eqref{id1.2}.

\medskip

\noindent Equation~\eqref{id3}: using equation~\eqref{eq3}, we have
\begin{eqnarray*}
\nabla_{0} \RRR_{j0} - \nabla_{j} \RRR_{00} = - g^{00}\RRR_{0j00}
\nabla_{0} f +\mu \big(\RRR_{0j}\nabla_{0} f - \RRR_{00}\nabla_{j}
f\big) - \lambda \big(g_{0j} \nabla_{0} f - g_{00} \nabla_{j} f \big)
= \mu |\nabla f|^{2} \RRR_{0j}\,.
\end{eqnarray*}

\medskip

\noindent Equation~\eqref{id3.2}: using again equation~\eqref{eq3}, we have
\begin{eqnarray*}
\nabla_{0} \RRR_{ij} - \nabla_{j} \RRR_{i0} &=& - g^{00}\RRR_{0ji0}
\nabla_{0} f +\mu \big(\RRR_{ij}\nabla_{0} f - \RRR_{i0}\nabla_{j}
f\big) - \lambda\mu \big(g_{ij} \nabla_{0} f - g_{i0} \nabla_{j} f
\big) \\
&=& \, -|\nabla f|^{4}\Big[
\tfrac{1}{n-2}\big(\RRR_{i0}g_{j0}+\RRR_{j0}g_{i0}-\RRR_{00}g_{ij}-\RRR_{ij}g_{00}\big)
-\tfrac{\RRR}{(n-1)(n-2)}\big(g_{j0}g_{i0}-g_{ij}g_{00}\big)\Big]\\
&&+ \mu |\nabla f|^{2}\RRR_{ij} - \lambda\mu|\nabla f|^{2}g_{ij}\\
&=& \tfrac{\mu(n-2)+1}{n-2}\RRR_{ij}|\nabla f|^{2} +
\tfrac{1}{n-2}\RRR_{00}|\nabla f|^{4} g_{ij} -
\tfrac{1}{(n-1)(n-2)}\RRR |\nabla f|^{2} g_{ij} - \lambda \mu |\nabla
f|^{2} g_{ij}\,.
\end{eqnarray*}
where in the second equality we have used the decomposition formula
for the Riemann tensor~\eqref{Weyl} and the fact that the Weyl curvature
part vanishes.
\end{proof}

Now, if we assume that the manifold is locally conformally flat, the
Cotton tensor is identically zero. Locally around every point where
$|\nabla f|\neq 0$, from Lemma~\ref{formulas}, we obtain
\begin{eqnarray*}
\CCC_{0j0} &=& \nabla_{0} \RRR_{j0} - \nabla_{j} \RRR_{00} -
\tfrac{1}{2(n-1)} \big( \nabla_{0} \RRR \, g_{j0} - \nabla_{j} \RRR \,
g_{00} \big)\\
&=& \mu |\nabla f|^{2} \RRR_{0j} +\tfrac{1-\mu}{n-1}|\nabla f|^{2} \,
\RRR_{0j}\\
&=& \tfrac{\mu(n-2)+1}{n-1}\, |\nabla f|^{2} \,\RRR_{0j}\,.
\end{eqnarray*}
Hence if $(M^{n},g)$ is locally conformally flat, we have that $\RRR_{0j}=0$ for every $j=1,\dots,n-1$, hence also
\begin{equation}\label{nullity}
\nabla_{j} \RRR = \nabla_{j} |\nabla f|^{2} = 0\,,
\end{equation}
where we have used again the previous lemma.
Hence, one has
\begin{eqnarray*}
\Gamma_{0j}^{0}&=& \frac{1}{2}g^{l0}(\partial_0g_{jl}+\partial_jg_{0l}-\partial_lg_{0j})\\
&=&\,\frac{1}{2}g^{00}(\partial_0g_{j0}+\partial_jg_{00})=0\,,\\
\Gamma_{00}^{j}&=&\,\frac{1}{2}g^{lj}(\partial_0g_{0l}+\partial_0g_{0l}-\partial_lg_{00})\\
&=&\,\frac{1}{2}g^{ij}(-\partial_ig_{00})=0\,,
\end{eqnarray*}
since $\partial_{j}g_{00}=\partial_{j}(|\nabla f|^{-2})=0$. An easy
computation shows that $\partial_{j} \RRR_{00} = 0$. Indeed, 
$$
\partial_{j} \RRR_{00} = \nabla_{j} \RRR_{00}+2\Gamma_{j0}^{0}
\RRR_{00} = \nabla_{j} \RRR_{00} = \nabla_{0} \RRR_{j0}=\partial_{0}
\RRR_{0j} - \Gamma_{00}^{i} \RRR_{ij}-\Gamma_{0j}^{0} \RRR_{00} = 0\,,
$$
where we used equations~\eqref{id3} and~\eqref{nullity}.
Now we want to show that the mean curvature of the level set
$\Sigma_{\rho}$ is constant on the level set. We recall that, since
$\nabla f / |\nabla f|$ is the unit normal 
vector to $\Sigma_{\rho}$, the second fundamental form $h$ 
verifies
\begin{equation}\label{IIform}
h_{ij} = -\frac{\nabla_{ij}^{2} f}{|\nabla f|} = \frac{\RRR_{ij} -
  \lambda g_{ij}}{|\nabla f|}\,,
\end{equation}
for $i,j=1,\dots,n-1$.
Thus, the mean curvature $\HHH$ of $\Sigma_{\rho}$
satisfies 
\begin{eqnarray}\label{mean}
\HHH = g^{ij}\,h_{ij} = \frac{\RRR -\RRR_{00}|\nabla f|^{2}- (n-1)\lambda
}{|\nabla f|} \,,
\end{eqnarray}
which clearly implies that the mean curvature is constant on $\Sigma_{\rho}$,
since all the quantities on the right hand side do. Now we want to
compute the components $\CCC_{ij0}$ of the Cotton tensor. Using
equations~\eqref{id2.2},~\eqref{id3.2} and~\eqref{nullity}, we have
\begin{eqnarray*}
\CCC_{ij0} &=& \nabla_{0} \RRR_{ij} - \nabla_{j} \RRR_{i0} -
\tfrac{1}{2(n-1)} \big( \nabla_{0} \RRR \, g_{ij} - \nabla_{j} \RRR \,
g_{i0} \big)\\
&=& \nabla_{0} \RRR_{ij} - \nabla_{j} \RRR_{i0} - \tfrac{1}{2(n-1)}
\nabla_{0} \RRR \, g_{ij}\\
&=& \tfrac{\mu(n-2)+1}{n-2}\RRR_{ij}|\nabla f|^{2} +
\tfrac{1}{n-2}\RRR_{00}|\nabla f|^{4} g_{ij} -
\tfrac{1}{(n-1)(n-2)}\RRR |\nabla f|^{2} g_{ij} - \lambda \mu |\nabla
f|^{2} g_{ij}+\\
&&\, - \tfrac{1}{2(n-1)}\Big[2 (1-\mu\big) \, |\nabla f|^{4}\RRR_{00}
-2(n-1)\mu\lambda|\nabla f|^{2} + 2\mu \RRR |\nabla
f|^{2}\Big]g_{ij}\\
&=& \tfrac{\mu(n-2)+1}{n-2}\Big(\RRR_{ij}
+\tfrac{1}{n-1}\RRR_{00}|\nabla f|^{2}g_{ij} - \tfrac{1}{n-1}\RRR
\,g_{ij}\Big)|\nabla f|^{2}\,.
\end{eqnarray*}
Finally, using the expression~\eqref{IIform} and~\eqref{mean}, we obtain
$$
\CCC_{ij0} = \tfrac{\mu(n-2)+1}{n-2} \big( h_{ij} - \tfrac{1}{n-1} \HHH
\,g_{ij} \big) |\nabla f|^{3}\,.
$$
Again, since all the components of the Cotton tensor vanish and we are
assuming that $\mu\neq\tfrac{1}{2-n}$, we obtain that 
\begin{eqnarray}\label{umbilic}
h_{ij}=\tfrac{1}{n-1} \HHH \, g_{ij}\,.
\end{eqnarray}
For any given $p\in \Sigma_{\rho}$, we suppose now to take orthonormal
coordinates centered at $p$, still denoted by 
$\theta^{1},\dots,\theta^{n-1}$. 
From the Gauss equation (see also~\cite[Lemma 3.2]{caochen} for a
similar argument), one can see that the sectional curvatures of
$(\Sigma_{\rho},g_{ij})$ at $p$ with the induced metric $g_{ij}$, are
given by
\begin{eqnarray*}
\RRR_{ijij}^{\Sigma} &=& \RRR_{ijij} + h_{ii}h_{jj} - h_{ij}^{2}\\
&=& \, \tfrac{1}{n-2}\big(\RRR_{ii}+\RRR_{jj}\big) -
\tfrac{1}{(n-1)(n-2)}\RRR + \tfrac{1}{(n-1)^{2}}\HHH^{2}\\
&=& \, \tfrac{2}{n-2}\RRR_{ii} - \tfrac{1}{(n-1)(n-2)}\RRR +
\tfrac{1}{(n-1)^{2}}\HHH^{2}\\
&=& \tfrac{2}{(n-1)(n-2)}\HHH|\nabla f| + \tfrac{2}{n-2}\lambda -
\tfrac{1}{(n-1)(n-2)}\RRR + \tfrac{1}{(n-1)^{2}}\HHH^{2}\,,
\end{eqnarray*}
for $i,j=1,\dots, n-1$, where in the second equality we made use of the decomposition formula for the Riemann tensor~\eqref{Weyl}, the locally conformally flatness of $g$ and of~\eqref{umbilic}. Since
all the terms on the right hand side are constant on $\Sigma_{\rho}$,
we obtain that the sectional curvatures of $(\Sigma_{\rho},g_{ij})$ 
are constant, which implies that $(M^{n},g)$ is locally a warped product
metric with fibers of constant curvature.
\qed

\begin{rem} Consider the manifold $(M^n,\widetilde{g})$ with the conformal metric $\widetilde{g}=e^{-\frac{2}{n-2}f}g$. Since the locally conformally flat property is conformally invariant this is a still locally conformally flat metric, hence its Cotton tensor is zero. Thus, from equations~\eqref{eq1} and ~\eqref{nullity} (this latter saying that the modulus of the gradient of $f$ is constant along any regular level set of $f$), it follows that its Ricci tensor has only two eigenvalues of multiplicities one and $(n-1)$, which are constant along the level sets of $f$. Indeed,
\begin{eqnarray*}
\Ric_{\widetilde{g}}&=&\Ric_{g} +\nabla^{2} f +\tfrac{1}{n-2} df \otimes
df +\tfrac{1}{n-2} \big(\Delta f - |\nabla f|^{2} \big) g\\
&=& \, \Bigl(\tfrac{1}{n-2}+\mu\Bigr) df \otimes
df + \tfrac{1}{n-2}\big(\Delta f - |\nabla f|^{2}+(n-2)\lambda\big)
e^{\frac{2}{n-2}f} \,\widetilde{g}\,.
\end{eqnarray*}
Then, arguing as in~\cite{mancat1} by means of splitting results for manifolds admitting a
  Codazzi tensor with only two distinct eigenvalues, we can conclude
  that $(M^n,\widetilde{g})$ is locally a warped product with
  $(n-1)$--dimensional fibers of constant curvature which are the
  level sets of $f$.\\
  By the structure of the conformal deformation this conclusion also holds for
  the original Riemannian manifold $(M^n,g)$.
\end{rem}

 It is well known that, if $(M^{n},g)$ is a compact locally
  conformally flat gradient shrinking Ricci soliton, then it has constant
  curvature (see~\cite{mantemin2}). As pointed out to us by the anonymous referee such a conclusion cannot be extended to quasi--Einstein metrics. Indeed, C. B\"{o}hm in \cite{bohm} has found Einstein metrics on $\SS^{k+1}\times \SS^{l}$ for $k,l\geq 2$ and $k+l\leq 8$ and these induce a quasi--Einstein metric on $\SS^{k+1}$ with $\mu=\frac{1}{l}$ and with the metric on $\SS^{k+1}$ being conformally flat (see also \cite{HePetWylie}).

In the complete, noncompact, case one would like to use Theorem~\ref{main} to have a
  classification of LCF quasi--Einstein manifolds (see~\cite{caochen} and~\cite{zhang} for steady and shrinking gradient Ricci
  solitons, respectively). Possibly one has to assume some curvature conditions as the nonnegativity of the curvature operator or of the Ricci tensor.

\bigskip

\begin{ackn} 
We wish to thank Andrea Landi and Alessandro Onelli for several
interesting comments on earlier versions of the paper.\\
The first three authors are partially supported by the Italian project FIRB--IDEAS ``Analysis and Beyond''.
\end{ackn}

\bibliographystyle{amsplain}
\bibliography{QuasiEinstein}

\end{document}